\documentclass[a4paper]{article}

\usepackage[english]{babel}
\usepackage[utf8]{inputenc}
\usepackage[margin=35mm]{geometry}
\usepackage[colorlinks,linkcolor=red,anchorcolor=green,citecolor=blue]{hyperref}
\usepackage{amsmath,amssymb,amsthm}
\usepackage{graphicx}
\usepackage{mathrsfs}
\usepackage{xcolor}

\numberwithin{equation}{section}

\newtheorem{theorem}{Theorem}[section]
\newtheorem{lemma}[theorem]{Lemma}
\newtheorem{proposition}[theorem]{Proposition}
\newtheorem{definition}[theorem]{Definition}

\newtheorem{remark}[theorem]{Remark}
\newtheorem{problem}[theorem]{Problem}

\newcommand{\R}{\mathbb{R}}

\newcommand{\N}{\mathbb{N}}
\newcommand{\Sp}{\mathbb{S}}

\begin{document}
\title{\textbf{The regularity of a semilinear elliptic system with quadratic growth of gradient}}

\author{
Weiyong He
\footnote{Department of Mathematics, University of Oregon, Eugene, OR 97403 (whe@uoregon.edu)} \qquad
\and
Ruiqi Jiang
\footnote{College of Mathematics and Econometrics, Hunan University, Changsha, 410082, P. R. China (jiangruiqi@hnu.edu.cn)}
}

\date{}
\maketitle

\begin{abstract}
In this paper, we study semilinear elliptic systems with critical nonlinearity of the form
\begin{equation}\label{sys01}
\Delta u=Q(x, u, \nabla u),
\end{equation}
for $u: \mathbb{R}^n\rightarrow \mathbb{R}^K$,  $Q$ has quadratic growth in $\nabla u$. Our work is motivated by elliptic systems for harmonic map and biharmonic map.
When $n=2$, such a system does not have smooth regularity in general for $W^{1, 2}$ weak solutions, by a well-known example of J. Frehse. Classical results of harmonic map, proved  by F. H\'elein (for $n=2$) and F. B\'ethuel (for $n\geq 3$), assert that a $W^{1, n}$ weak solution of harmonic map is always smooth.
We extend B\'ethuel's result to general system \eqref{sys01},  that a $W^{1, n}$ weak solution of the system  is smooth for $n\geq 3$. For a fourth order semilinear elliptic system with critical nonlinearity which extends biharmonic map, we prove a similar result, that a $W^{2, n/2}$ weak solution of such system is always smooth, for $n\geq 5$. We also construct various examples, and these examples show that our regularity results are optimal in various sense.
\end{abstract}

\textbf{Key Words:} harmonic and biharmonic maps; quadratic growth of gradient; \\
\hspace*{2.7cm} elliptic systems with critical nonlinearity; Lorentz space

\textbf{AMS subject classifications:} 35J47, 35J48, 35J91

\section{Introduction}
The theory of regularity of harmonic map is very influential in geometric analysis and partial differential equations.
In this paper we study two semilinear elliptic systems which are partly motivated by regularity theory of harmonic maps and biharmonic maps.

Firstly, let us consider the following elliptic system of second order.
Let $\Omega$ be an open domain in Euclidean space $\R^n$ and let $N$ be a compact Riemannian manifold embedding in some Euclidean space $\R^K$.
Consider a vector valued function $u: \Omega \subset \R^n \rightarrow N \subset \R^K$ satisfying the following equation
\begin{align}\label{eqn:main1}
\Delta u = Q(x,u,\nabla u),
\end{align}
where $Q: \R^n \times \R^K \times \R^{nK} \rightarrow \R^K$ satisfies
\begin{align}\label{cond:main1}
|Q(x,y,z)| \leq C |z|^2,\quad (x,y,z) \in \R^n \times \R^K \times \R^{nK}
\end{align}
for some positive constant $C$.
Note that if $u\in W^{1, 2}$, the notion of weak solution for \eqref{eqn:main1} is well-defined. We shall consider \eqref{eqn:main1} in the setting of weak solutions.
A well-known example of \eqref{eqn:main1} is  harmonic map, which satisfies the system \eqref{eqn:main1} with very special structure on the righthand side.
For example, if we take $N=S^{K-1}\subset \R^K$, then the harmonic map equation reads
\begin{equation}
\Delta u^i=-|\nabla u|^2 u^i , \quad i=1,\cdots, K.
\end{equation}
We recall a well-known result of H\'elein \cite{Helein}, that a $W^{1, 2}$ weakly harmonic map for $n=2$ is always smooth.
H\'elein's result relies crucially on the structure of the system of harmonic map, for which the righthand side has special algebraic structures.
A general system as in \eqref{eqn:main1} does not share smooth regularity as harmonic maps (for $n=2$).
In fact, Frehse \cite{Frehse} has constructed an example ($n=2$), that the system \eqref{eqn:main1} has a $W^{1, 2}\cap L^\infty$ solution but it is not continuous at $x=0$.

For higher dimensions ($n\geq 3$), H\'elein's result for weakly harmonic map is no longer true.
To be more precise, the map: $x|x|^{-1}: B^3\rightarrow S^2$ is a weakly harmonic map and has a singular point at the origin.
In general it is relatively easy to construct weakly harmonic map which has a singular set of dimension $n-3$.
Surprisingly Rivi\`ere \cite{Riviere} constructed a weakly harmonic map from $B^n\rightarrow S^k$ which are not continuous everywhere in $B^n$.
Hence there is no partial regularity theory for a general weakly harmonic map. An interesting result in this direction is as follows, proved by B\'ethuel \cite{Bethuel}(Theorem I.2) as a consequence of his regularity result of \emph{stationary harmonic maps}.

\begin{theorem}[B\'ethuel]Let $u$ be a weakly harmonic map in $W^{1, n}(M, N)$ ($\text{dim}_\mathbb{R} M=n$), then $u$ is smooth ($n\geq 3$).
\end{theorem}

\begin{remark}
It should be noted that $p=n$ ($n\geq 3$) for the assumption $u\in W^{1, p}(M, N)$ is critical.
Of course, for $p>n$, $u\in C^{\alpha}$ by Sobolev embedding implies the regularity of the system (\ref{eqn:main1}).
On the other hand, we can find a singular map $u$ (see the example in Sec.\ref{sec:example2}) which is in $W^{1,p}$ for any $p \in [2,n)$ but solves the system (\ref{eqn:main1}) in weakly sense.
\end{remark}

We are interested in the system \eqref{eqn:main1} when $n\geq 3$. A simple observation indicates that there is subtle difference in the system between the case $n=2$ and $n\geq 3$. When $n=2$, the righthand side would be in $L^1$ and elliptic regularity hardly provides extra information. But when $n\geq 3$, the righthand side would be in $L^{n/2}$, and elliptic regularity ($L^p$ theory) instantly implies that $u\in W^{2, n/2}$.  One cannot obtain directly that $u$ is continuous via Sobolev embedding of course (since the system is ``critical"), but $W^{2, n/2}$ is indeed a finer space than $W^{1, n}$.
In other words, we gain a little edge automatically via the system itself and the assumption $u\in W^{1, n}$, for $n\geq 3$. This difference actually leads to a totally different story of the system \eqref{eqn:main1} for $n\geq 3$. We have the following,

\begin{theorem}\label{thm:main1}
If $u$ is a weak solution to equation (\ref{eqn:main1}) and $u \in W^{1,n}(\Omega)$ for $n \geq 3$, then $u \in C^\alpha(\Omega, \R^K)$ for some $\alpha\in (0, 1)$. Moreover, if $Q \in C^\infty(\R^n \times \R^K \times \R^{nK})$, then $u\in C^\infty(\Omega, \R^K)$.
\qed
\end{theorem}

With the assumption $u\in W^{1, n}$, we can say that the system \eqref{eqn:main1} is \emph{most critical} when $\text{dim}_\mathbb{R} M=2$ (the righthand side is merely integrable). In this most critical situation, the system does not share smooth regularity in view of Frehse's example. Hence the algebraic structure of the righthand side is then crucial to obtain further regularity, as indicated by H\'elein's result of harmonic map (see also Rivi\`ere's in \cite{Riviere2}).  This phenomenon appears certainly in  much broader circumstances. As an indication, we consider the following elliptic system of fourth order, for $u\in W^{2, 2}(\Omega)$,
\begin{align}\label{eqn:main2}
\Delta^2 u = Q_1(x,u,\nabla u, \nabla^2 u) + \nabla \cdot Q_2(x,u,\nabla u, \nabla^2 u)
\end{align}
where $Q_{i}: \R^n \times \R^K \times \R^{nK} \times \R^{n^2K} \rightarrow \R^K$, $i=1,2$ satisfy
\begin{align}
\Big|Q_1(x,u,\nabla u, \nabla^2 u) \Big| &\leq C
\Big[ |\nabla u|^4 + |\nabla u|^2 |\nabla^2 u| + |\nabla^2 u|^2  \Big], \label{cond:main2-1}\\
\Big|Q_2(x,u,\nabla u, \nabla^2 u) \Big| &\leq C
|\nabla u| |\nabla^2 u|, \label{cond:main2-2}
\end{align}
for some positive constant $C$. The system of biharmonic maps is one of the most interesting example of \eqref{eqn:main2}. In terms of our observation, we can say that when $n=4$ the system \eqref{eqn:main2} is \emph{most critical}. It is not hard to imagine that the system does not share smooth regularity in general in this most critical situation. Indeed, we show that a straightforward extension of Frehse's example to dimension four gives an example of a weak solution of a system \eqref{eqn:main2}, but it is not continuous at $x=0$. We have the following,

\begin{theorem}\label{thm:main3}
Consider the map $u=(u_1, u_2): \Omega \rightarrow \Sp^1 \hookrightarrow \R^2$ defined by
\begin{align}
u_1(x)=\sin (\log (\log |x|^{-1})), \quad u_2(x)=\cos (\log (\log |x|^{-1}))
\end{align}
where $\Omega=\{x \in \R^4 : \, |x| \leq \exp(-2) \}$. Then $u$ lies in $W^{2, 2}(\Omega)\cap L^\infty(\Omega)$ but is not continuous at point $x=0$, and $u$ is a $W^{2, 2}$-weak solution of the following system
\begin{align}
\Delta^2 u_1 &=\left( R_1^2+R_2^2 \right) \frac{2(u_1+u_2)}{1+|u|^2}
- \frac{20 u_1}{1+|u|^2}|\nabla u|^4 ,\\
\Delta^2 u_2 &=\left( R_1^2+R_2^2 \right) \frac{2(u_2-u_1)}{1+|u|^2}
- \frac{20 u_2}{1+|u|^2}|\nabla u|^4, 
\end{align}
where
\begin{align}
R_1 &= \Delta u_1 + \frac{2(u_1+u_2)}{1+|u|^2} |\nabla u|^2 ,\\
R_2 &= \Delta u_2 + \frac{2(u_2-u_1)}{1+|u|^2} |\nabla u|^2 .
\end{align}
\end{theorem}

As a comparison,  a weakly $W^{2, 2}$ biharmonic map in $\Omega\subset \mathbb{R}^4$ is always smooth, proved by Chang-Wang-Yang \cite{CWY} (when the target is a sphere) and Wang \cite{Wang} (for general targets). Needless to say, the special structure of biharmonic map  plays an essential role for its smooth regularity. On the other hand, it is easy to see that when $n\geq 5$, the general system \eqref{eqn:main2} is indeed \emph{less critical} with the assumption $u\in W^{2, n/2}$, similar as our observation above. Thus one might speculate that such a system always shares smooth regularity with the assumption $u\in W^{2, n/2}$.  We confirm this speculation as follows.

\begin{theorem}\label{thm:main2}
If $u$ is a weak solution to equation (\ref{eqn:main2}) and $u \in W^{2,n/2}(\Omega)$ for $n \geq 5$, then $u \in C^\alpha(\Omega, \R^K)$ for some $\alpha\in (0, 1)$. Moreover, if $Q_i \in C^\infty(\R^n \times \R^K \times \R^{nK} \times \R^{n^2K})$, $i=1,2$, then $u\in C^\infty(\Omega, \R^K)$.
\qed
\end{theorem}

\begin{remark}
Similarly, $p=n/2$ ($n\geq 5$) for the assumption $u\in W^{2, p}$ is critical.
The example of the singular map $u\in W^{2, p}$ for any $p \in [2, n/2)$ satisfying the system (\ref{eqn:main2}) in the distribution sense can be found in Sec.\ref{sec:example3}.
\end{remark}

We should mention that Wang \cite{Wang} has considered a special case of the system \eqref{eqn:main2}.
\begin{theorem}[Wang, C.Y., \cite{Wang}Theorem B]
Consider the following fourth order PDE with borderline nonlinearity,
\begin{equation}\label{eqn:wangsys}
\Delta^2 u=Q(x, u, \nabla u), \quad x\in \Omega\subset \mathbb{R}^4,
\end{equation}
where $Q:\Omega \times \mathbb{R}^K\times \mathbb{R}^{4K}$ satisfies
\[
Q(x, y, p)\leq C|p|^4, \quad \forall (x, y, p)\in \Omega \times \mathbb{R}^K \times \mathbb{R}^{4K}.
\]
Suppose $u\in W^{2, 2}(\Omega, \mathbb{R}^K)$ is a weak solution of \eqref{eqn:wangsys},
then there exists an $\alpha\in (0, 1)$ such that $u\in C^\alpha(\Omega, \mathbb{R}^K)$. Moreover, if $Q$ is smooth, then $u$ is smooth.
\end{theorem}

We shall emphasize that Wang's result holds for $n=4$ (and $u\in W^{2, n/2}=W^{2, 2}$), and the system \eqref{eqn:wangsys} is a rather special case of \eqref{eqn:main2}. His result does not really violate our observation, but strengthens it: the most critical term involved in the righthand side of \eqref{eqn:wangsys} is $|\nabla u|^4$, but $u\in W^{2, 2}$ and $W^{2, 2}$ is a finer space than $W^{1, 4}$. If one checks Wang's proof carefully, the finer structure of $W^{2, 2}$ (than $W^{1, 4}$) is indeed used crucially to prove his Theorem B in \cite{Wang}.
One key point to the proof of Theorem \ref{thm:main2}  is to explore the finer structure of $W^{2, n/2}$ carefully than $W^{1, n}$. We use the Lorentz spaces to explore this finer structure and our proof is partly motivated by the proof of Wang, \cite{Wang}[Theorem B].

\medskip

The paper is organized as follows. In Section \ref{sec:preli}, we gather various facts concerning Lorentz space. In Section \ref{sec:proof-main1} and \ref{sec:proof-main2}, we prove Theorem \ref{thm:main1} and Theorem \ref{thm:main2}, respectively. Section \ref{sec:cexample} is devoted to the construction of singular maps and to the proof of Theorem \ref{thm:main3}. The final section contains some discussions and problems that we plan to study in the future.

\section{Preliminary results}\label{sec:preli}
In this section, we gather some facts about Lorentz space that will be used later.
First, let us recall the celebrated theorem of Morrey.
\begin{theorem}[\textbf{Morrey}] \label{thm:Morrey}
Let $1<p<\infty$, $0<\alpha<1$, and $\Omega \subset \mathbb{R}^n$ be a bounded domain with appropriately smooth boundary (such as $\partial \Omega \in C^{1,\alpha}$).
If $u \in W^{1,p}(\Omega)$ and for any $x \in \Omega$ and $0 <\rho< diam \Omega$, there holds
\begin{align}
\int_{\Omega_{\rho}(x)} |\nabla u(y)|^p dy \leq C \rho^{n-p+p\alpha},
\end{align}
where $\Omega_{\rho}(x)=\Omega \cap B_{\rho}(x)$, then $u \in C^\alpha(\overline{\Omega})$.
\qed
\end{theorem}

Then we need some results about Lorentz spaces which play an important role in our proof. All the results presented in this section can be found in \cite{CR}, \cite{ONeil}, \cite{Tartar} and \cite{Ziemer}. For the convenience of readers, we recall the definition of Lorentz spaces.

For a measurable function $f: \Omega \rightarrow \R$, the distribution function of $f$ is defined by $D_f(\lambda)= \mu \{x \in \Omega : |f(x)| > \lambda \}$ and the decreasing rearrangement of $f$ is the function $f^*: [0, \infty) \rightarrow [0, \infty]$ defined by
\begin{align*}
f^*(t) =\inf \{ \lambda \geq 0: D_f(\lambda) \leq t\}.
\end{align*}

\begin{definition}
For $1 \leq p < \infty$ and $1 \leq q \leq \infty$, the Lorentz space $\mathbf{L}^{p,q}(\Omega)$ is defined as
\begin{align}
\mathbf{L}^{p,q}(\Omega)=\{ f \,\mbox{is a measurable function on $\Omega$}: \, \|f \|_{\mathbf{L}^{p,q}(X)}<\infty\}
\end{align}
which is defined by
\begin{align}
\|f \|_{\mathbf{L}^{p,q}(\Omega)}=
\left\{
\begin{aligned}
&\left(\int_0^\infty \left(t^{\frac{1}{p}} f^{**}(t) \right)^q \frac{dt}{t} \right)^{\frac{1}{q}},
&1&\leq p <\infty, 1\leq q < \infty, \\
&\sup_{t>0} t^{\frac{1}{p}} f^{**}(t), &1&\leq p <\infty, q = \infty,
\end{aligned}
\right.
\end{align}
where
\begin{align*}
f^{**}(t)=\frac{1}{t} \int_0^t f^*(s) ds.
\end{align*}
\qed
\end{definition}


\begin{theorem}\label{thm:L(p,q-)}
~
\begin{enumerate}
  \item If $1<p<\infty$, then $\mathbf{L}^{p,p}(\Omega)=L^p(\Omega)$ and
      \begin{align}
      \| f\|_{L^p(\Omega)} \leq \|f \|_{\mathbf{L}^{p,p}(\Omega)} \leq \frac{p}{p-1} \| f\|_{L^p(\Omega)}.
      \end{align}
  \item If $1<p <\infty$ and $1\leq q < s \leq \infty$, then $\mathbf{L}^{p,q}(\Omega)\hookrightarrow \mathbf{L}^{p,s}(\Omega)$ and
      \begin{align}
      \|f \|_{\mathbf{L}^{p,s}(\Omega)} \leq \left( \frac{q}{p}\right)^{\frac{1}{q}-\frac{1}{s}}
      \|f \|_{\mathbf{L}^{p,q}(\Omega)}.
      \end{align}
  \item If $1\leq p <\infty$ and $1\leq q  \leq \infty$, then $\mathbf{L}^{p,q}(\Omega)\hookrightarrow \mathbf{L}^{p,\infty}(\Omega)$ and
      \begin{align}
      \|f \|_{\mathbf{L}^{p,\infty}(\Omega)}\leq \left( \frac{q}{p}\right)^{\frac{1}{q}}
      \|f \|_{\mathbf{L}^{p,q}(\Omega)}
      \end{align}
\end{enumerate}
\qed
\end{theorem}

\begin{theorem}\label{thm:L(p-,q)}
Suppose $f \in \mathbf{L}^{p_1, q_1}(\Omega)$, $g\in \mathbf{L}^{p_2,q_2}(\Omega)$ satisfy
$\frac{1}{p}=\frac{1}{p_1}+\frac{1}{p_2} <1$ and $\frac{1}{q} \leq \frac{1}{q_1}+\frac{1}{q_2} $, then we have $fg \in \mathbf{L}^{p, q}(\Omega)$, and
\begin{align}
\|fg \|_{\mathbf{L}^{p, q}(\Omega)} \leq
\|f \|_{\mathbf{L}^{p_1, q_1}(\Omega)} \cdot \|g \|_{\mathbf{L}^{p_2,q_2}(\Omega)}.
\end{align}
In particular, if $\Omega$ is bounded and $1<r<p<\infty$, $1 \leq q,s \leq \infty$, then we have $\mathbf{L}^{p,q}(\Omega)\hookrightarrow \mathbf{L}^{r,s}(\Omega)$ and
\begin{align}
\|f \|_{\mathbf{L}^{r,s}(\Omega)} \leq C
|\Omega|^{\frac{1}{r}-\frac{1}{p}}\|f \|_{\mathbf{L}^{p,q}(\Omega)},
\end{align}
where $C=C(n,p,q,r,s)>0.$
\qed
\end{theorem}

\begin{theorem}
For $1<a,c<\infty$ and $1 \leq b,d \leq \infty$, if $f \in \mathbf{L}^{p_1,q_1}(\R^n), g \in\mathbf{L}^{p_2,q_2}(\R^n)$ satisfying $\frac{1}{p_1}+\frac{1}{p_2}-1=\frac{1}{p}>0$, $\frac{1}{q_1}+\frac{1}{q_2}=\frac{1}{q} \leq 1$, then $f*g \in \mathbf{L}^{p,q}(\R^n)$, and
\begin{align}
\|f*g \|_{\mathbf{L}^{p,q}(\R^n)} \leq C \|f \|_{\mathbf{L}^{p_1,q_1}(\R^n)}
\|g \|_{\mathbf{L}^{p_2,q_2}(\R^n)}.
\end{align}

\qed
\end{theorem}

\begin{proposition}\label{prop:L(p,infty)}
Suppose $f(x)=|x|^{-s}$ for $x \in \R^n \setminus \{0\}$ with $0<s < n$, then
$$f \in \mathbf{L}^{\frac{n}{s}, \infty}(\R^n).$$
\end{proposition}
\begin{proof}
Since
\begin{align*}
D_f(\lambda)=\mu \{x \in \R^n : |f(x)|> \lambda\}=b_n \lambda^{-\frac{n}{s}}
\end{align*}
where $b_n$ is the volume of the unit ball in $\R^n$, we have
\begin{align*}
f^*(t)=\inf \{ \lambda \geq 0 : D_f(\lambda) \leq t\}
= \left(\frac{b_n}{t} \right)^{\frac{s}{n}}.
\end{align*}
It follows that
\begin{align*}
t^{\frac{s}{n}} f^{**}(t)=\frac{n}{n-s} (b_n)^{\frac{s}{n}},
\end{align*}
which implies that
$$f(x)=\frac{1}{|x|^{s}} \in \mathbf{L}^{\frac{n}{s},\infty}(\R^n).$$
\end{proof}

\noindent \textbf{Convention:} For the sake of simplicity, we always denote by $B_\theta (x)$ the open ball of radius $\theta >0$ centered at a point $x$ in $\R^n$. Sometimes we may denote $B_\theta(0)$ by $B_\theta$.

\section{Proof of Theorem \ref{thm:main1}} \label{sec:proof-main1}

First, let us recall the fundamental solution $\Phi(x)$ of Laplace's equation which is defined by
\begin{align}\label{func:Green-2nd}
\Phi(x)=
\left\{
\begin{aligned}
&-\frac{1}{2\pi} \log |x|, \quad &n&=2\\
&\frac{1}{n(n-2)b_n} \frac{1}{|x|^{n-2}},\quad &n&\geq 3,
\end{aligned}
\right.
\end{align}
where $b_n$ denotes the volume of the unit ball in $\R^n$.

\begin{lemma}\label{lem:Green-2nd}
For $n \geq 3$, we have
\begin{align}
\nabla \Phi(x) \in \mathbf{L}^{\frac{n}{n-1},\infty}(\R^n).
\end{align}
\end{lemma}

\begin{proof}
By a simple calculation, we have
\begin{align*}
|\nabla \Phi(x)| \leq C(n) \frac{1}{|x|^{n-1}},
\end{align*}
where $C(n)$ is a positive number only independent of $n$.

It follows from Proposition \ref{prop:L(p,infty)} that
\begin{align}
\|\nabla \Phi(x) \|_{\mathbf{L}^{\frac{n}{n-1},\infty}(\R^n)}
\leq C(n) \||x|^{1-n}\|_{\mathbf{L}^{\frac{n}{n-1},\infty}(\R^n)} <\infty,
\end{align}
which is desired.
\end{proof}

\begin{lemma}\label{lem:theta-2nd}
Suppose that $\phi \in C^\infty (B_1)$ is a harmonic function, i.e., $\Delta \phi(x) =0$ in the unit ball of $\R^n$ for $n \geq 3$. Then for any $x \in B_{\frac{1}{4}}$ and $ \theta \in (0, \frac{1}{4})$, there holds
\begin{align}\label{ineq:theta-2nd}
\|\nabla \phi \|_{\mathbf{L}^{n,\infty}(B_\theta(x))} \leq C \theta \|\nabla \phi \|_{\mathbf{L}^{n,\infty}(B_1)}
\end{align}
where $C$ is only dependent of $n$.
\qed
\end{lemma}

\begin{proof}
Unless otherwise specified, we assume $C$ to be some positive constant only dependent of $n$ in the following proof.

Since $\phi(x)$ is a harmonic function, we have, for all $i=1, \cdots, n$, $\frac{\partial \phi}{\partial x^i}$ is also a harmonic function in $B_1$. In order to prove inequality (\ref{ineq:theta-2nd}), it suffices to prove
\begin{align}
\|\phi \|_{\mathbf{L}^{n,\infty}(B_\theta(x))} \leq C \theta \|\phi \|_{\mathbf{L}^{n,\infty}(B_1)},
\end{align}
for any $x \in B_{\frac{1}{4}}$ and $ \theta \in (0, \frac{1}{4})$.

Since $\phi(x)$ is a harmonic function, we have
\begin{align*}
\|\phi \|_{L^\infty(B_{\frac{1}{2}})} \leq C \|\phi \|_{L^2(B_1)}.
\end{align*}
Due to the fact $n \geq 3$, it follows from Theorem \ref{thm:L(p,q-)} and Theorem \ref{thm:L(p-,q)} that
\begin{align}\label{ineq:theta-2nd-1}
\|\phi \|_{L^\infty(B_{\frac{1}{2}})} \leq C \|\phi \|_{L^2(B_1)}
\leq C \|\phi \|_{\mathbf{L}^{n,\infty}(B_1)}.
\end{align}
At the same time, by Theorem \ref{thm:L(p,q-)}, we have
\begin{align}\label{ineq:theta-2nd-2}
\|\phi \|_{\mathbf{L}^{n,\infty}(B_\theta(x))}
\leq  \|\phi \|_{\mathbf{L}^{n,n}(B_\theta(x))}
\leq C \|\phi \|_{L^n(B_\theta(x))}
\leq C \theta \|\phi \|_{L^\infty (B_\theta(x))}.
\end{align}
For all $x \in B_{\frac{1}{4}}$ and $ \theta \in (0, \frac{1}{4})$, we have $B_\theta(x) \subset B_{\frac{1}{2}}$. Thus, combing (\ref{ineq:theta-2nd-1}) and (\ref{ineq:theta-2nd-2}), we obtain the desired result.

\end{proof}

\begin{lemma}\label{lem:2nd-epsilon}
There exist $\epsilon_0>0$ and $\theta_0>0$ such that if $u \in W^{1,n}(B_1, \R^K)$ $(n\geq 3)$ is a weak solution to the equation (\ref{eqn:main1}) satisfying
\begin{align*}
\int_{B_1} |\nabla u|^n dx \leq \epsilon_0^n,
\end{align*}
then we have
\begin{align}\label{ineq:2nd-epsilon}
\|\nabla u \|_{\mathbf{L}^{n,\infty}(B_{\theta_0})}
\leq \frac{1}{2} \|\nabla u \|_{\mathbf{L}^{n,\infty}(B_1)}.
\end{align}
\end{lemma}

\begin{proof}
Unless otherwise specified, we assume $C$ to be some positive constant only dependent of $n$ in the following proof.

For $u \in W^{1,n}(B_1, \R^K)$, there exists a extension $\widetilde{u} \in  W^{1,n}(\R^n, \R^K)$ such that
\begin{align*}
\widetilde{u}|_{B_1}=u|_{B_1} - L,
\end{align*}
where $L=\frac{1}{|B_1|}\int_{B_1} u(x) dx$, and
\begin{align*}
\| \nabla \widetilde{u}\|_{L^n(\R^n)}
& \leq C \|\nabla u \|_{L^n(B_1)} \leq C \epsilon_0 \notag \\
\| \nabla \widetilde{u}\|_{\mathbf{L}^{n,\infty}(\R^n)}
& \leq C \|\nabla u \|_{\mathbf{L}^{n,\infty}(B_1)}.
\end{align*}
Define a function
\begin{align*}
v(x)= -\int_{\R^n} \Phi(x-y) Q(y,\widetilde{u}(y)+L,\nabla \widetilde{u}(y)) dy
\end{align*}
where $\Phi(x)$ is the fundamental function defined by (\ref{func:Green-2nd}) for $n \geq 3$.
Then we have that
\begin{align}
\| \nabla v\|_{\mathbf{L}^{n,\infty}(\R^n)}
&\leq C \| \nabla \Phi\|_{\mathbf{L}^{\frac{n}{n-1},\infty}(\R^n)}
       \| Q\|_{\mathbf{L}^{\frac{n}{2},\infty}(\R^n)} \notag \\
&\leq C \| |\nabla \widetilde{u}|^2\|_{\mathbf{L}^{\frac{n}{2},\infty}(\R^n)} \notag \\
&\leq C \| \nabla \widetilde{u}\|^2_{\mathbf{L}^{n,\infty}(\R^n)} \notag \\
&\leq C \| \nabla \widetilde{u}\|_{\mathbf{L}^{n,n}(\R^n)}
        \| \nabla \widetilde{u}\|_{\mathbf{L}^{n,\infty}(\R^n)} \notag \\
&\leq C \epsilon_0 \| \nabla \widetilde{u}\|_{\mathbf{L}^{n,\infty}(\R^n)} \notag \\
&\leq C \epsilon_0 \| \nabla u\|_{\mathbf{L}^{n,\infty}(B_1)} \label{ineq:2nd-epsilon-1}
\end{align}
and
\begin{align*}
\Delta (u-v) =0, \quad \mbox{in} \,\, B_{1}.
\end{align*}
Thus, by Lemma \ref{lem:theta-2nd} and inequality (\ref{ineq:2nd-epsilon-1}), we have, for all $\theta \in (0,\frac{1}{4})$
\begin{align}
\|\nabla (u-v) \|_{\mathbf{L}^{n,\infty}(B_\theta)}
& \leq C \theta \, \|\nabla (u-v) \|_{\mathbf{L}^{n,\infty}(B_1)} \notag \\
& \leq C \theta \left(\|\nabla u\|_{\mathbf{L}^{n,\infty}(B_1)}
         + \|\nabla v\|_{\mathbf{L}^{n,\infty}(B_1)}\right) \notag \\
& \leq C \theta \left(\|\nabla u\|_{\mathbf{L}^{n,\infty}(B_1)}
         + \|\nabla v\|_{\mathbf{L}^{n,\infty}(\R^n)}\right) \notag \\
& \leq C \theta \left(\|\nabla u\|_{\mathbf{L}^{n,\infty}(B_1)}
         + \epsilon_0 \|\nabla u\|_{\mathbf{L}^{n,\infty}(B_1)}\right).\label{ineq:2nd-epsilon-2}
\end{align}
It follows that
\begin{align}
\| \nabla u\|_{\mathbf{L}^{n,\infty}(B_\theta)}
&\leq \| \nabla (u-v)\|_{\mathbf{L}^{n,\infty}(B_\theta)}
     + \| \nabla v\|_{\mathbf{L}^{n,\infty}(B_\theta)} \notag \\
&\leq C \theta \|\nabla u\|_{\mathbf{L}^{n,\infty}(B_1)}
     + C \epsilon_0 \|\nabla u\|_{\mathbf{L}^{n,\infty}(B_1)}.
\end{align}
Hence, we can choose $\theta=\theta_0  \in (0,\frac{1}{4})$ and $\epsilon_0>0$ small enough such that inequality (\ref{ineq:2nd-epsilon}) holds.
\end{proof}

\bigskip
\noindent \textbf{Proof of Theorem \ref{thm:main1}}
Now we are in the position to prove the main result. For any precompact domain $U \subset \subset \Omega$, we have that for any $\epsilon_0>0$ there exists a $\rho_0\in (0, dist(\overline{U}, \partial \Omega))$ such that
\begin{align*}
\sup_{x \in \widetilde{U}} \int_{B_{\rho_0}(x)} |\nabla u(y)|^n dy \leq \epsilon_0^n.
\end{align*}
It is easy to verify that for any fix point $x^* \in \overline{U}$, $u_{x^*,\rho_0}(y)=u(x^*+ \rho_0 y)$ satisfies the same system of $u(x)$ in $B_1$, i.e.,
\begin{align*}
\Delta u_{x^*,\rho_0}(y)= \widetilde{Q}( y, u_{x^*,\rho_0}, \nabla u_{x^*,\rho_0})
\end{align*}
which satisfies
\begin{align*}
|\widetilde{Q}( y, u_{x^*,\rho_0}(y), \nabla u_{x^*,\rho_0})|
\leq C | \nabla u_{x^*,\rho_0}|^2
\end{align*}
for some positive constant $C$ independent of $u$.

Since
$\| \nabla u_{x^*, \rho_0} \|_{L^n(B_1)}
=\| \nabla u \|_{L^n(B_{\rho_0}(x^*))} \leq \epsilon_0$,
it follows from Lemma \ref{lem:2nd-epsilon} that there exists a $\theta_0 \in (0, \frac{1}{4})$ such that for any $x^* \in \overline{U}$, we have
\begin{align}
\| \nabla u_{x^*, \rho_0}\|_{L^{n,\infty}(B_{\theta_0})}
=\| \nabla u\|_{L^{n,\infty}(B_{\theta_0 \rho_0}(x^*))} \notag \\
\leq \frac{1}{2} \| \nabla u_{x^*, \rho_0}\|_{L^{n,\infty}(B_1)}
=\frac{1}{2} \| \nabla u\|_{L^{n,\infty}(B_{\rho_0}(x^*))}.\label{ineq:proof-2nd-1}
\end{align}
By iterating above inequality (\ref{ineq:proof-2nd-1}), for any $l \in \N$ and $x^* \in \overline{U}$, we have
\begin{align*}
\| \nabla u\|_{L^{n,\infty}(B_{\theta_0^l \rho_0}(x^*))}
\leq \frac{1}{2^l} \| \nabla u\|_{L^{n,\infty}(B_{\rho_0}(x^*))}.
\end{align*}
By replacing $\theta_0^l \rho_0$ by $\rho$, we have
\begin{align*}
\| \nabla u\|_{L^{n,\infty}(B_{\rho}(x^*))}
\leq C \rho^{\alpha_0} \| \nabla u\|_{L^{n,\infty}(B_{\rho_0}(x^*))}
\leq C \rho^{\alpha_0},
\end{align*}
where $\alpha_0=\frac{\log 2}{\log \theta_0^{-1}} \in (0,1)$ and $C$ is only dependent of $\rho_0$, $\theta_0$ and $\epsilon_0$.
\\
Hence, by Theorem \ref{thm:L(p-,q)}, we have that for any $1<p<n$, $x^* \in \overline{U}$ and $0<\rho<\rho_0$, there holds
\begin{align}
\|\nabla u\|_{L^p(B_{\rho}(x^*))}
&\leq \|\nabla u\|_{\mathbf{L}^{p,p}(B_{\rho}(x^*))} \notag \\
&\leq C |B_{\rho}(x^*)|^{\frac{1}{p}-\frac{1}{n}}
     \|\nabla u\|_{\mathbf{L}^{n,\infty}(B_{\rho}(x^*))} \notag \\
&\leq C \rho^{\frac{n-p}{p}+\alpha_0}.
\end{align}
It follows from Theorem \ref{thm:Morrey} that $ u \in C^{\alpha_0}(\overline{U}, M)$ which implies that $ u \in C^{\alpha_0}(\Omega, M)$. If $Q \in C^\infty(\Omega \times \R^K \times \R^{nK})$ and $u\in C^\alpha$, it is well-known that such a system enjoys smooth regularity; see \cite{CWY0}[Section 3] for example. Hence, the proof is completed.
\qed

\section{Proof of Theorem \ref{thm:main2}}\label{sec:proof-main2}
Since the proof of Theorem \ref{thm:main2} is similar as that of Theorem \ref{thm:main1}, we only give the key points which are different from the counterparts in Section \ref{sec:proof-main1}. Unless otherwise specified in this section, we always assume the dimension of domain is not less than four, i.e., $n \geq 4$.

\medskip

First, let us recall the fundamental solution $\Psi(x)$ of $\Delta^2$ on $\R^n$ which is defined by
\begin{align}\label{func:Green-4th}
\Psi(x)=
\left\{
\begin{aligned}
&c_4 \log |x|, \quad &n&=4\\
&c_n \frac{1}{|x|^{n-4}},\quad &n&\geq 5,
\end{aligned}
\right.
\end{align}
where $c_n$ ($n\geq 4$) are constants only dependent of $n$.
It follows from Proposition \ref{prop:L(p,infty)} that, for $n \geq 5$, we have, for $k=1,2,3$,
\begin{align}
\nabla^k \Psi(x) \in \mathbf{L}^{\frac{n}{n-4+k},\infty}(\R^n).
\end{align}

Then, we need some Caccioppoli-type inequality for smooth biharmonic functions to prove Lemma \ref{lem:theta-4th}.
\begin{lemma}\label{lem:cacci}
Suppose that $\phi \in C^\infty(B_1)$ is a biharmonic function, i.e., $\Delta^2 \phi =0$ in $B_1$. Then for any $\theta \in (0,1)$, we have
\begin{align}
\int_{B_\theta} |\nabla \phi|^2 dx + \int_{B_\theta} |\nabla^2 \phi|^2 dx \leq C \int_{B_1} |\phi|^2 dx,
\end{align}
where $C=C(n,\theta)>0$, provided that $\phi \in L^2(B_1)$.
\end{lemma}
\begin{proof}
We refer the reader to \cite{Shen}[Lemma 2.1] for the proof of Lemma \ref{lem:cacci}.
\end{proof}

Now it is time for us to prove the similar result as Lemma \ref{lem:theta-2nd} for smooth biharmonic functions.
\begin{lemma}\label{lem:theta-4th}
Suppose that $\phi \in C^\infty(B_1) \cap W^{2,\frac{n}{2}}(B_1)$ is a biharmonic function, i.e., $\Delta^2 \phi =0$ in the unit ball of $\R^n$ for $n \geq 5$. Then for any $x \in B_{\frac{1}{4}}$ and $\theta \in (0, \frac{1}{4})$, there holds
\begin{align}\label{ineq:theta-4th}
\|\nabla \phi \|_{\mathbf{L}^{n,\infty}(B_{\theta}(x))} +
\|\nabla^2 \phi \|_{\mathbf{L}^{\frac{n}{2},\infty}(B_{\theta}(x))}
\leq C \theta
\left(
\|\nabla \phi \|_{\mathbf{L}^{n,\infty}(B_1)} +
\|\nabla^2 \phi \|_{\mathbf{L}^{\frac{n}{2},\infty}(B_1)}
\right),
\end{align}
where $C$ is only dependent of $n$.
\end{lemma}

\begin{proof}
Throughout the proof, $C$ always stands for  positive constants only dependent of $n$.
Since $\phi$ is a biharmonic function, we have, for all $i, j=1,\cdots,n$, $\frac{\partial \phi}{\partial x^i}$ and $\frac{\partial^2 \phi}{\partial x^i \partial x^j}$ are also biharmonic functions. In order to prove inequality (\ref{ineq:theta-4th}), it suffices to prove
\begin{align}
\| \phi \|_{\mathbf{L}^{p, \infty}(B_\theta(x))} &\leq C \theta
\| \phi \|_{\mathbf{L}^{p, \infty}(B_1)}
\end{align}
for all $x \in B_{\frac{1}{4}}$, $\theta \in (0,\frac{1}{4})$ and $ p \in [\frac{n}{2},n]$.

\medskip
By the standard elliptic estimate, we have, for $n \geq 5$ and $p \in [\frac{n}{2},n]$,
\begin{align}
\| \phi \|_{L^\infty(B_{\frac{1}{2}})}
&\leq C
\left(
\| \phi \|_{L^2(B_{\frac{5}{8}})} + \| \Delta \phi \|_{L^\infty(B_{\frac{5}{8}})}
\right) \notag \\
&\leq C
\left(
\| \phi \|_{L^2(B_{\frac{5}{8}})} + \| \Delta \phi \|_{L^2(B_{\frac{6}{8}})}
\right) \notag \\
&\leq C
\left(
\| \phi \|_{L^2(B_{\frac{5}{8}})} + \| \phi \|_{L^2(B_{\frac{7}{8}})}
\right) \notag \\
&\leq C
\| \phi \|_{L^2(B_1)} \notag \\
& \leq C
\| \phi \|_{\mathbf{L}^{p, \infty}(B_1)},\label{ineq:theta-4th-1}
\end{align}
where in the second inequality we used the fact $\Delta \phi$ is a harmonic function, in the third inequality Lemma \ref{lem:cacci}, and in the last inequality Theorem \ref{thm:L(p-,q)}.

At the same time, we have, for $n \geq 5$ and $p \in [\frac{n}{2},n]$,
\begin{align}\label{ineq:theta-4th-2}
\| \phi \|_{\mathbf{L}^{p, \infty}(B_\theta(x))}
\leq C \| \phi \|_{\mathbf{L}^{p, p}(B_\theta(x))}
\leq C \| \phi \|_{L^p(B_\theta(x))}
\leq C \theta^{\frac{n}{p}} \| \phi \|_{L^\infty(B_\theta(x))}.
\end{align}
For any $x \in B_{\frac{1}{4}}$ and $\theta \in (0,\frac{1}{4}$, we have $B_\theta(x) \subset B_{\frac{1}{2}}$, which, combing the inequality (\ref{ineq:theta-4th-1}) and (\ref{ineq:theta-4th-2}), yields that
\begin{align}
\| \phi \|_{\mathbf{L}^{p, \infty}(B_\theta(x))}
\leq C \theta^{\frac{n}{p}} \| \phi \|_{L^\infty(B_\theta(x))}
\leq C \theta \| \phi \|_{L^\infty(B_{\frac{1}{2}})}
\leq C \theta \| \phi \|_{\mathbf{L}^{p, \infty}(B_1)},
\end{align}
which is the desired conclusion.
\end{proof}

\medskip

\begin{lemma}\label{lem:4th-epsilon}
There exist $\epsilon_0>0$ and $\theta_0>0$ such that if $u \in W^{2,\frac{n}{2}}(B_1, \R^K)$ $(n\geq 5)$ is a weak solution to the equation (\ref{eqn:main2}) satisfying
\begin{align}
\|\nabla u\|_{L^n(B_1)} + \|\nabla^2 u\|_{L^{\frac{n}{2}}(B_1)} \leq \epsilon_0,
\end{align}
then we have
\begin{align}\label{ineq:4th-epsilon}
\|\nabla u \|_{\mathbf{L}^{n,\infty}(B_{\theta_0})}
+ \|\nabla^2 u \|_{\mathbf{L}^{\frac{n}{2},\infty}(B_{\theta_0})}
\leq \frac{1}{2}
\left(
\|\nabla u \|_{\mathbf{L}^{n,\infty}(B_1)}
+ \|\nabla^2 u \|_{\mathbf{L}^{\frac{n}{2},\infty}(B_1)}
\right).
\end{align}
\end{lemma}

\begin{proof}
For simplicity, we always denote by $C$ the positive constant independent of $u$.\\
For $u \in W^{2,\frac{n}{2}}(B_1, \R^K)$, there exists a extension $\widetilde{u} \in  W^{2,\frac{n}{2}}(\R^n, \R^K)$ such that
\begin{align}
\widetilde{u}|_{B_1}=u|_{B_1}-L
\end{align}
where $L=\frac{1}{|B_1|}\int_{B_1}u(x)dx$, and
\begin{align*}
\| \nabla \widetilde{u}\|_{L^n(\R^n)}
&\leq C \|\nabla u \|_{L^n(B_1)} \leq C \epsilon_0, \\
\| \nabla \widetilde{u}\|_{\mathbf{L}^{n,\infty}(\R^n)}
&\leq C \|\nabla u \|_{\mathbf{L}^{n,\infty}(B_1)}\leq C \epsilon_0, \\
\| \nabla^2 \widetilde{u}\|_{L^{\frac{n}{2}}(\R^n)}
&\leq C
\left(
\|\nabla u \|_{L^n(B_1)}+ \|\nabla^2 u \|_{L^{\frac{n}{2}}(B_1)}
\right)
\leq C \epsilon_0, \\
\| \nabla^2 \widetilde{u}\|_{\mathbf{L}^{\frac{n}{2},\infty}(\R^n)}
&\leq C
\left(
\|\nabla u \|_{\mathbf{L}^{n,\infty}(B_1)} +\|\nabla^2 u \|_{\mathbf{L}^{\frac{n}{2},\infty}(B_1)}
\right)
\leq C \epsilon_0.
\end{align*}
Note that we will use the following fact:
\begin{align*}
\Big|Q_1(x, u, \nabla u, \nabla^2 u) \Big| \leq C
\left(
|\nabla u|^4 + |\nabla u|^2 |\nabla^2 u| + |\nabla^2 u|^2
\right)
\leq C
\left(
|\nabla u|^4 + |\nabla^2 u|^2
\right).
\end{align*}
\\
Define two functions
\begin{align}
v_1(x) &= -\int_{\R^n} \Psi(x-y) Q_1(y,\widetilde{u}(y)+L,\nabla \widetilde{u}(y), \nabla^2 \widetilde{u}(y)) dy  \\
v_2(x) &= \int_{\R^n} \nabla \Psi(x-y) Q_2(y,\widetilde{u}(y)+L,\nabla \widetilde{u}(y), \nabla^2 \widetilde{u}(y)) dy
\end{align}
where $\Psi(x)$ is the fundamental function defined by (\ref{func:Green-4th}) for $n \geq 5$. Then we have that
\begin{align}
& \| \nabla v_1\|_{\mathbf{L}^{n, \infty}(\R^n)}
+ \| \nabla^2 v_1\|_{\mathbf{L}^{\frac{n}{2}, \infty}(\R^n)} \notag \\
\leq &
\left\|\int_{\R^n}  \nabla \Psi(x-y) Q_1(y) dy  \right\|_{\mathbf{L}^{n, \infty}(\R^n)}
+ \left\|\int_{\R^n} \nabla^2 \Psi(x-y) Q_1(y) dy \right\|_{\mathbf{L}^{\frac{n}{2}, \infty}(\R^n)} \notag \\
\leq & C
\left(
\| \nabla \Psi \|_{\mathbf{L}^{\frac{n}{n-3}, \infty}(\R^n)}
+ \| \nabla^2 \Psi \|_{\mathbf{L}^{\frac{n}{n-2}, \infty}(\R^n)}
\right)
\| |\nabla \widetilde{u}|^4 + |\nabla^2 \widetilde{u}|^2\|_{\mathbf{L}^{\frac{n}{4}, \infty}(\R^n)} \notag \\
\leq & C
\left(
\| \nabla \widetilde{u}\|^4_{\mathbf{L}^{n, \infty}(\R^n)}
+ \|\nabla^2 \widetilde{u}\|^2_{\mathbf{L}^{\frac{n}{2}, \infty}(\R^n)}
\right) \notag \\
\leq & C \epsilon_0
\left(
\| \nabla u\|_{\mathbf{L}^{n, \infty}(B_1)}
+ \|\nabla^2 u \|_{\mathbf{L}^{\frac{n}{2}, \infty}(B_1)}
\right),\label{ineq:4th-epsilon-1}
\end{align}
and
\begin{align}\label{ineq:4th-epsilon-2}
& \| \nabla v_2\|_{\mathbf{L}^{n, \infty}(\R^n)}
+ \| \nabla^2 v_2\|_{\mathbf{L}^{\frac{n}{2}, \infty}(\R^n)} \notag \\
\leq &
\left\|\int_{\R^n}  \nabla^2 \Psi(x-y) Q_2(y) dy \right\|_{\mathbf{L}^{n, \infty}(\R^n)}
+\left\|\int_{\R^n} \nabla^3 \Psi(x-y) Q_2(y) dy \right\|_{\mathbf{L}^{\frac{n}{2}, \infty}(\R^n)} \notag \\
\leq & C
\left(
\| \nabla^2 \Psi \|_{\mathbf{L}^{\frac{n}{n-2}, \infty}(\R^n)}
+ \| \nabla^3 \Psi \|_{\mathbf{L}^{\frac{n}{n-1}, \infty}(\R^n)}
\right)
\| |\nabla \widetilde{u}| \cdot |\nabla^2 \widetilde{u}|\|_{\mathbf{L}^{\frac{n}{3}, \infty}(\R^n)} \notag \\
\leq & C
\left(
\| \nabla \widetilde{u}\|_{\mathbf{L}^{n, \infty}(\R^n)}
\cdot \|\nabla^2 \widetilde{u}\|_{\mathbf{L}^{\frac{n}{2}, \infty}(\R^n)}
\right) \notag \\
\leq & C \epsilon_0
\| \nabla u\|_{\mathbf{L}^{n, \infty}(B_1)}.
\end{align}
Moreover, since
\begin{align}
\Delta^2 (u-v_1-v_2) = 0, \quad \mbox{in $B_{1}$},
\end{align}
by Lemma \ref{lem:theta-4th} and inequalities (\ref{ineq:4th-epsilon-1}) and (\ref{ineq:4th-epsilon-2}), we have, for all $\theta \in (0,\frac{1}{4})$,
\begin{align}
& \| \nabla(u-v_1-v_2)\|_{\mathbf{L}^{n, \infty}(B_{\theta})}
+ \| \nabla^2 (u-v_1-v_2)\|_{\mathbf{L}^{\frac{n}{2}, \infty}(B_{\theta})} \notag \\
\leq & C \theta
\left(
\| \nabla(u-v_1-v_2)\|_{\mathbf{L}^{n, \infty}(B_{1})}
+ \| \nabla^2 (u-v_1-v_2)\|_{\mathbf{L}^{\frac{n}{2}, \infty}(B_{1})}
\right) \notag \\
\leq & C \theta
\big(
\| \nabla u\|_{\mathbf{L}^{n, \infty}(B_{1})}
+ \| \nabla v_1\|_{\mathbf{L}^{n, \infty}(B_{1})}
+\| \nabla v_2\|_{\mathbf{L}^{n, \infty}(B_{1})}  \notag \\
&+ \| \nabla^2 u\|_{\mathbf{L}^{\frac{n}{2}, \infty}(B_{1})}
+ \| \nabla^2 v_1\|_{\mathbf{L}^{\frac{n}{2}, \infty}(B_{1})}
+\| \nabla^2 v_2\|_{\mathbf{L}^{\frac{n}{2}, \infty}(B_{1})}
\big) \notag \\
\leq & C \theta
\left(
\| \nabla u\|_{\mathbf{L}^{n, \infty}(B_{1})}
+  \| \nabla^2 u\|_{\mathbf{L}^{\frac{n}{2}, \infty}(B_{1})}
\right).
\end{align}
It follows that
\begin{align*}
& \| \nabla u\|_{\mathbf{L}^{n, \infty}(B_{\theta})}
+ \| \nabla^2 u\|_{\mathbf{L}^{\frac{n}{2}, \infty}(B_{\theta})} \\
\leq &
\| \nabla(u-v_1-v_2)\|_{\mathbf{L}^{n, \infty}(B_{\theta})}
+ \| \nabla^2 (u-v_1-v_2)\|_{\mathbf{L}^{\frac{n}{2}, \infty}(B_{\theta})} \\
& + \| \nabla v_1 \|_{\mathbf{L}^{n, \infty}(B_{\theta})}
+ \| \nabla v_2 \|_{\mathbf{L}^{n, \infty}(B_{\theta})}
+ \| \nabla^2 v_1\|_{\mathbf{L}^{\frac{n}{2}, \infty}(B_{\theta})}
+ \| \nabla^2 v_2\|_{\mathbf{L}^{\frac{n}{2}, \infty}(B_{\theta})} \\
\leq & C \theta
\left(
\| \nabla u\|_{\mathbf{L}^{n, \infty}(B_{1})}
+  \| \nabla^2 u\|_{\mathbf{L}^{\frac{n}{2}, \infty}(B_{1})}
\right)
+ C \epsilon_0
\left(
\| \nabla u\|_{\mathbf{L}^{n, \infty}(B_{1})}
+  \| \nabla^2 u\|_{\mathbf{L}^{\frac{n}{2}, \infty}(B_{1})}
\right).
\end{align*}
Hence, we can choose $\theta=\theta_0 \in (0,\frac{1}{4})$ and $\epsilon_0>0$ small enough such that inequality (\ref{ineq:4th-epsilon}) holds.
\end{proof}

The rest proof of Theorem \ref{thm:main2} is exactly similar as that of Theorem \ref{thm:main1}. Thus we omit it. We would like to refer the reader to \cite{CWY} for the higher order regularity from the H\"older continuity estimate.

\section{The Counterexample}\label{sec:cexample}
\subsection{The singular solution to the system (\ref{eqn:main2}) for \texorpdfstring{$n=4$}{n=4}}\label{sec:example1}

In this subsection, we plan to prove Theorem \ref{thm:main3}. To be more exact, we will give a detailed calculation to show that there exists a singular map $u \in W^{2,2}$ for $n=4$, first introduced by Frehse \cite{Frehse}, satisfying the equation (\ref{eqn:main2}) in the sense of distribution.

\medskip

The map $u=(u_1, u_2): \Omega \rightarrow \Sp^1 \hookrightarrow \R^2$ is defined by
\begin{align}\label{eg:4th}
u_1(x)=\sin (\log (\log |x|^{-1})), \quad u_2(x)=\cos (\log (\log |x|^{-1}))
\end{align}
where $\Omega=\{x=(x^1,x^2,x^3,x^4) \in \R^4 : \, |x| \leq \exp(-2) \}$. Of course, $u$ lies in $L^\infty(\Omega)$ but is singular at point $x=0$.

\medskip

First, Let us check that $u\in W^{2,2}(\Omega)$. For simplicity, Let
$$f(x)=\log |x|^{-1},\quad x \in \R^4,\quad 0<|x|\leq \exp(-2).$$
By a simple calculation, we have, for $i, j=1,2,3,4$,
\begin{align*}
\frac{\partial u_1}{\partial x^i} &= \frac{f_i}{f}\cos \log f, \qquad
\frac{\partial u_2}{\partial x^i} = -\frac{f_i}{f} \sin \log f, \\
\frac{\partial^2 u_1}{\partial x^i \partial x^j} &= - \frac{f_i f_j}{f^2} \sin \log f
+ \frac{f_{ij} f - f_i f_j}{f^2} \cos \log f, \\
\frac{\partial^2 u_2}{\partial x^i \partial x^j} &= - \frac{f_i f_j}{f^2} \cos \log f
- \frac{f_{ij} f - f_i f_j}{f^2} \sin \log f,
\end{align*}
where
\begin{align*}
f_i=\frac{\partial f}{\partial x^i}=- \frac{x^i}{|x|^2}, \qquad
f_{ij}= \frac{\partial^2 f}{\partial x^i \partial x^j}= -\frac{\delta_{ij}}{|x|^2}+ \frac{2 x^i x^j}{|x|^4}.
\end{align*}
It follows that
\begin{align}
|\nabla u|^2  &= \frac{|\nabla f|^2}{f^2}=\frac{1}{|x|^2 f^2}, \\
|\nabla^2 u|^2 & \leq C \left( \frac{|\nabla f|^4}{f^4} + \frac{|\nabla^2 f|^2}{f^2}\right)
\leq C \left( \frac{1}{|x|^4 f^4} + \frac{1}{|x|^4 f^2} \right),
\end{align}
where $C$ is a positive constant independent of $u$. By direct computation, we obtain
\begin{eqnarray}
\frac{1}{|x|^2 f^2}, \frac{1}{|x|^4 f^4}, \frac{1}{|x|^4 f^2} \in L^1(\Omega),
\end{eqnarray}
which implies $ u \in W^{2,2}(\Omega)$ by Fubini's theorem.

\medskip

We are now in the position to show $u$ is a weak solution to the following system
\begin{align}
\Delta^2 u_1 &=\left( R_1^2+R_2^2 \right) \frac{2(u_1+u_2)}{1+|u|^2}
- \frac{20 u_1}{1+|u|^2}|\nabla u|^4 ,\label{eqn:eg-4th-1}\\
\Delta^2 u_2 &=\left( R_1^2+R_2^2 \right) \frac{2(u_2-u_1)}{1+|u|^2}
- \frac{20 u_2}{1+|u|^2}|\nabla u|^4, \label{eqn:eg-4th-2}
\end{align}
where
\begin{align}
R_1 &= \Delta u_1 + \frac{2(u_1+u_2)}{1+|u|^2} |\nabla u|^2 ,\\
R_2 &= \Delta u_2 + \frac{2(u_2-u_1)}{1+|u|^2} |\nabla u|^2 .
\end{align}
Since $u=(u_1,u_2)$ is smooth on $\Omega \setminus \{0 \}$, one can check that $u$ satisfies the system (\ref{eqn:eg-4th-1})-(\ref{eqn:eg-4th-2}) in the classic sense in the domain $\Omega \setminus \{0 \}$. Moreover, it follows from $u \in L^\infty(\Omega) \cap W^{2,2}(\Omega) \cap C^\infty(\Omega \setminus \{0 \})$ and Fubini's theorem that $u$ is just a weak solution of the system (\ref{eqn:eg-4th-1})-(\ref{eqn:eg-4th-2}).

For the convenience of  reader we give the details.
Since $u \in C^\infty(\Omega \setminus \{0 \})$ and satisfies the system (\ref{eqn:eg-4th-1})-(\ref{eqn:eg-4th-2}) in classic sense in $\Omega \setminus \{0 \}$, it suffices to show the following equality holds for all $\varphi \in C^\infty_0(\Omega, \R^2)$
\begin{align}
\int_{\Omega}\Delta u(x) \Delta \varphi(x)=\int_{\Omega} \Delta^2 u(x) \varphi(x) dx.
\end{align}
In fact, since $\Delta u \in W^{2,2}(\Omega, \R^2)$, we have, for all $\varphi \in C^\infty_0(\Omega, \R^2)$,
\begin{align*}
\int_{\Omega} \Delta u \cdot \Delta \varphi dx
&=\int_{\R^4} \Delta u \cdot \Delta \varphi dx \\
&=\int_{\R} dx^4 \int_{\R^3} \Delta u \cdot (\partial_1^2 +\partial_2^2+ \partial_3^2)\varphi dx'
+ \int_{\R} dx^4 \int_{\R^3} \Delta u \cdot \partial_4^2 \varphi dx' \\
&=\int_{\R} dx^4 \int_{\R^3} (\partial_1^2 +\partial_2^2+ \partial_3^2)\Delta u \cdot \varphi dx'
+ \int_{\R} dx^4 \int_{\R^3} \Delta u \cdot \partial_4^2 \varphi dx' \\
&=\int_{\R} dx^4 \int_{\R^3} (\partial_1^2 +\partial_2^2+ \partial_3^2)\Delta u \cdot \varphi dx'
+  \int_{\R^3} dx' \int_{\R} \Delta u \cdot \partial_4^2 \varphi dx^4  \\
&=\int_{\R} dx^4 \int_{\R^3} (\partial_1^2 +\partial_2^2+ \partial_3^2)\Delta u \cdot \varphi dx'
+  \int_{\R^3} dx' \int_{\R} \partial_4^2 \Delta u \cdot  \varphi dx^4  \\
& = \int_{\R^4} \Delta^2 u \cdot \varphi dx,
\end{align*}
where $x=(x^1, x^2, x^3, x^4)=(x', x^4)$. In above second equality and fourth equality we used Fubini's theorem, and in the third and fifth equality integration by parts for $x^4 \neq 0$ and $x' \neq 0$ respectively.

\medskip

Finally, one can easily show that the right hand side of the equation (\ref{eqn:eg-4th-1}) and (\ref{eqn:eg-4th-2}) satisfy the condition (\ref{cond:main2-1}). In conclusion, the map $u$, defined by (\ref{eg:4th}), is just what we want.

\subsection{The singular solution to the system (\ref{eqn:main1}) for
\texorpdfstring{$n \geq 3$}{n>=3}}\label{sec:example2}
In this subsection, we give a singular map $u$ which lies in $W^{1,p}$ for all $p \in [2, n)$ and satisfies the system \eqref{eqn:main1} in the distribution sense.

The map $u=(u_1, u_2): \Omega \rightarrow \Sp^1 \hookrightarrow \R^2$ is defined by
\begin{align}
u_1(x)=\sin \Big((2-n)\log |x| \Big), \quad u_2(x)=\cos \Big((2-n)\log |x| \Big)
\end{align}
where $\Omega=\{x\in \R^n : \, |x| \leq \exp(-2) \}$ ($n \geq 3$).

Following the same procedure of Sec.\ref{sec:example1}, one can check that $u\in W^{1,p}(\Omega) \cap L^\infty (\Omega)$ for any $p \in [2, n)$ solves the following system in the weakly sense
\begin{align}
\Delta u_1 &= -  \frac{2(u_1+u_2)}{1+|u|^2}|\nabla u|^2,  \\
\Delta u_2 &= - \frac{2(u_2-u_1)}{1+|u|^2}|\nabla u|^2.
\end{align}

\subsection{The singular solution to the system (\ref{eqn:main2}) for
\texorpdfstring{$n \geq 5$}{n >=5}}\label{sec:example3}
In this subsection, we give a singular map $u$ which lies in $W^{2,p}$ for all $p \in [2, n/2)$ and satisfies the system \eqref{eqn:main2} in the distribution sense.

The map $u=(u_1, u_2): \Omega \rightarrow \Sp^1 \hookrightarrow \R^2$ is defined by
\begin{align}
u_1(x)=\sin \Big((4-n)\log |x| \Big), \quad u_2(x)=\cos \Big((4-n)\log |x| \Big)
\end{align}
where $\Omega=\{x\in \R^n : \, |x| \leq \exp(-2) \}$ ($n \geq 5$).

Following the same procedure of Sec.\ref{sec:example1}, one can check that $u\in W^{2,p}(\Omega) \cap L^\infty (\Omega)$ for all $p \in [2, n/2)$ solves the following system in the weakly sense
\begin{align}
\Delta^2 u_1 &=\left( R_1^2+R_2^2 \right) \frac{2(u_1+u_2)}{1+|u|^2}
+ \frac{4 u_2}{1+|u|^2}|\nabla u|^4 + (R_2-R_1) |\nabla u|^2,\\
\Delta^2 u_2 &=\left( R_1^2+R_2^2 \right) \frac{2(u_2-u_1)}{1+|u|^2}
- \frac{4 u_1}{1+|u|^2}|\nabla u|^4 - (R_1+R_2) |\nabla u|^2, 
\end{align}
where
\begin{align}
R_1 &= \Delta u_1 + \frac{2(u_1+u_2)}{1+|u|^2} |\nabla u|^2 ,\\
R_2 &= \Delta u_2 + \frac{2(u_2-u_1)}{1+|u|^2} |\nabla u|^2 .
\end{align}

\section{Discussions and further problems}\label{sec:discussions}

We use Lorentz spaces and Morrey's result to prove H\"older continuity. Consider the following Morrey's subnorm for $n\geq 3$, $2\leq p\leq n$,
\[
M_p(x_0, r)(u)=r^{p-n}\int_{B_{r}(x_0)} |\nabla u|^p dx
\]
A simple application of H\"older inequality implies that
\[
M_p(x_0, r)(u)\leq C_n\|\nabla u\|_{L^n(B_r(x_0))}^p.
\]
We can ask the following,

\begin{problem}\label{p1}Suppose $u\in W^{1, p}(B_1, \mathbb{R}^k)$ is a weak solution of \eqref{eqn:main1} for $B_1\subset \mathbb{R}^n$. If we assume further that for any $x_0\in B_1$, and $r<\text{dist}(x_0, \partial B_1)$,
\[
\lim_{r\rightarrow 0}M_p(x_0, r)(u)=0
\]
Is $u$ H\"older continuous in $B_{1/2}$ for $p \in (2, n]$? Certainly we assume $n\geq 3$.

\end{problem}
We use the Laplacian operator in the system for simplicity. In general one can replace $\Delta u$ by $a^{ij}D^2_{ij}u$, given that $a^{ij}D^2_{ij}$ defines a smooth uniformly elliptic operator.
We ask the following,

\begin{problem}\label{p2}
Consider the following system
\[
D_i(a^{ij}D_j u)=Q(x, u, \nabla u)
\]
such that $a^{ij}\in W^{1, n}\cap L^\infty$ defines a uniformly elliptic operator and $Q$ has quadratic gradient growth. Is a $W^{1, n}$ weak solution necessarily H\"older continuous (for $n\geq 3$ of course)?

\end{problem}
We believe the answer to both problems should be affirmative by similar observations; while it would be  interesting to study Problem \ref{p1} for $p=2$. Theorem \ref{eqn:main1} corresponds to Problem \ref{p1} when $p=n$ and it seems that similar method could work for $p\in (2, n]$. We shall consider these problems elsewhere.

\medskip

\textbf{Acknowledgement.} The research of Weiyong He was partially supported by NSF grant 1611797. The research of Ruiqi Jiang was supported by a grant from the Fundamental Research Funds for the Central Universities.

\end{document}